\documentclass[11 pt]{article}
\usepackage{amsmath,amscd}
\usepackage{amstext}
\usepackage{amssymb,epsfig}
\usepackage{amsthm}
\usepackage{amsfonts}
\usepackage{fancybox, floatrow}
\usepackage{graphicx}
\usepackage{hyperref}
\usepackage[all]{xy}
\usepackage{subeqnarray}

\usepackage[utf8]{inputenc}      
\usepackage{tikz}

\newtheorem{mainth}{Theorem}
\newtheorem*{theoNoNum}{Theorem}
\newtheorem{proposition}{Proposition}
\newtheorem{theorem}[proposition]{Theorem}

\newtheorem{lemma}[proposition]{Lemma}
\newtheorem{corollary}[proposition]{Corollary}

\newtheorem{definition}[proposition]{Definition}
\newtheorem{remark}[proposition]{Remark}

\def\N{{\mathbb N}}

\def\R{{\mathbb R}}

\def\K{{\mathbb K}}

\newcommand {\CA}{{\mathcal A}}

\newcommand {\CL}{{\mathcal L}}

\newcommand {\CT}{{\mathcal T}}

\def\s{\sigma}

\def\1{ {\hbox{{\it 1}} \!\! I} }

\def\be{\beta}

\def\8{\infty}

\def\disp{}

\usepackage{float,color}

\newcommand{\wh}{\widehat}

\newcommand{\BBone}{{1\!\!1}}

\theoremstyle{definition}

\begin{document}

\title{Freezing phase transition for the Thue-Morse subshift}
\author{
Nicolas Bédaride\footnote{
    Aix Marseille Université, CNRS,
     I2M, UMR 7373, 13331 Marseille, France.
    Email: \texttt{nicolas.bedaride@univ-amu.fr}} 
    \and 
        Julien Cassaigne\footnote{
     CNRS,
     Aix Marseille Université, I2M UMR 7373, 13331 Marseille, France.
    Email: \texttt{julien.cassaigne@math.cnrs.fr}}
    \and 
       Pascal Hubert\footnote{
    Aix Marseille Université, CNRS,
     I2M, UMR 7373, 13331 Marseille, France.
    Email: \texttt{pascal.hubert@univ-amu.fr}}
    \and 
    Renaud Leplaideur
      \footnote{Université de la Nouvelle-Calédonie 145, Avenue James Cook - BP R4 Nouméa 98 851 - Cedex Nouvelle Calédonie
    \texttt{renaud.leplaideur@unc.nc}}
    \thanks{The two first authors were partially supported by 
ANR Project IZES ANR-22-CE40-0011}
}

\date{}
\maketitle

\begin{abstract}
On the full shift on two symbols, we consider the potential defined by $V(x) = \frac{1}{n}$ where $n$ denotes the longest common prefix between the infinite word $x$ and an element of the subshift associated to the Thue-Morse substitution. Given a non negative real number $\beta$, the pressure function is 
$P(\be):=\sup\left\{h_{\mu}+\be\int V\,d\mu\right\},$ where the supremum is taken over all shift invariant probabilities $\mu$ on the full shift and $h_{\mu}$ is the Kolmogorov entropy.
We prove that there is a freezing phase transition for the potential $V$: For $\beta$ large enough, the pressure $P(\be)$ is equal to zero.
Similar results were previously published by Bruin and Leplaideur in \cite{BL2}, \cite{Bruin-Leplaid-13}  but their proofs contained significant gaps and required substantial clarification. 
\end{abstract}

\section{Introduction}

Given a symbolic dynamical system $(X,\s)$, a potential  $V:X\to\R$, and a non-negative real number $\beta$, the pressure function is defined by 
$$P(\be):=\sup\left\{h_{\mu}+\be\int V\,d\mu\right\},$$
where the supremum is taken over all $\s$-invariant probabilities $\mu$ on $X$ and $h_{\mu}$ is the Kolmogorov entropy. For an invariant measure $\mu$ its pressure is $h_\mu+\beta V$.
An equilibrium state at $\be$ is a measure which realizes the maximum in the previous formula. In our settings, a phase transition at $\be_{0}$ is a point where the pressure function is not analytic. These two notions are important and are well studied.
Given a subshift of finite type and a H\"older potential $V$, Sinai, Ruelle and Bowen (\cite{sinai,Ruelle.78,Bow.08}) proved in their seminal works that there is a unique equilibrium state and that the pressure is an analytic function. Moreover if $V$ has summable variations, then there is an unique equilibrium state and the pressure function is $\mathcal C^1$, see \cite{Walters.75}.
Nevertheless, this is not always the case for other subshifts. Examples of potentials with phase transitions do exist. The first result in this direction was obtained by  Hofbauer    
see \cite{Hofbauer-77}. In the symbolic setting, and in all the following, the space $X$ is the full shift on the alphabet $\{0,1\}$. The value of the potential is equal to $\frac{1}{n+1}$ for each infinite word $x$ starting by $0^n1$ and is equal to $0$ for the word $0^\infty$. This potential is a continuous function on the full shift but not a H\"older function for the classical distance on the full shift.
The integer $n$ is the longest commun prefix between $x$ and the word $0^\infty$. 
For $\beta$ large enough, the only equilibrium state is the Dirac measure at $0^\infty$, which means that the pressure is equal to zero for $\beta$ large enough. This phenomenon is called a freezing phase transition \footnote{In physical terms, $\beta$ is the inverse of the temperature, a freezing phase transition corresponds to a situation when, at finite temperature, the system is frozen.}. The survey by Chazottes and Keller \cite{Chazottes-Keller-08}  gives a nice overview of these matters.
\\

 Bruin and Leplaideur \cite{BL2}, \cite{Bruin-Leplaid-13} wanted to generalise Hofbauer's result by replacing the trivial subshift $\{0^\infty\}$ by subshifts that can be seen as toy models of quasicrystals, namely substitution subshifts. Note that the space $X$ will remain the full shift, we will introduce other subshifts in order to define the potential. Given a subshift 
$\K$ of the full shift on a finite alphabet, if a word $x$ does not belong to $\K$, the value of the potential is $\frac{1}{n^\alpha}$ where $n$ denotes the longest common prefix between $x$ and an element of $\K$ and $\alpha$ is a positive real parameter.
They consider the problem when $\K$ is a subshift, and, to be more specific,when $\K$ is the subshift generated by a fixed point of a primitive substitution like the Thue-Morse substitution or the Fibonacci substitution. In this case, they prove that $\alpha = 1$ is a critical case in the sens that before and after $1$ the results are different: for $\alpha > 1$, there is no phase transition, and for $\alpha< 1$, there is a phase transition. The most delicate situation being $\alpha = 1$. Nevertheless, the proofs of the different results have gaps that cannot be fixed even with hard work. The technical point is a computation where the number of words of length $n$ belonging to the language of $\K$ is not taken into account. The present paper gives a positive answer to this question in the case of the Thue-Morse substitution. As in Bruin-Leplaideur, the proof is based on the understanding of induced potentials but the combinatorial analysis is of a different nature and involves new ideas in word combinatorics. Roughly speaking we are interested in studying substitutions 'from the outside and not from the inside'. We establish the link between
the occurrences of bispecial words in a given infinite word and the behaviour of the distance to $\K$ along the orbit of this infinite word under the shift. In particular we emphasize the importance of the notion of {\it accident}.

Although a more formal version of this result will be presented later in the paper, our main result is as follows:
\begin{theoNoNum}
Let $X$ be the full shift on a two-letter alphabet. Let $\mathrm{TM}$ be the subshift associated to the Thue-Morse substitution and $V(x) = \frac{1}{n}$ where $n$ denotes the longest common prefix between $x\in X$ and an element of $\mathrm{TM}$. There exists $\beta_0$ such that, for $\beta \geq \beta_0$, the only equilibrium state is the zero entropy measure supported  on $\mathrm{TM}$ and the pressure is equal to 0 which means that there is a freezing phase transition. 
\end{theoNoNum}

We only have an upper bound on the value of $\beta_0$. A lower bound was obtained by Ishaq and Leplaideur  \cite{Ishaq-Leplaideur-22} for the Fibonacci substitution. The exact value of $\beta_0$ remains an open problem.  \\

{\bf Related works.} Maldonado and Salgado \cite{Maldonado-Salgado} proved analogous results for potentials associated to the middle-third Cantor set. Nevertheless the structure of the middle-third Cantor set is considerably simpler than the one we study. 
While this work was being done, Kucherenko and Quas \cite{Kucherenko-Quas-21} proved very general and impressive results for potentials defined on  two-sided subshift. The existence of phase transitions is related to the regularity of the potential. The speed $\log(n)/n$ is the critical one (the meaning of $n$ is more or less the same as in our approach), and for comparison our speed is $1/n$. In the classical theory, the study of potentials on two-sided shifts is reduced to studying potentials on one-sided shift using a coboundary, see \cite{Walters.75}. This is only possible if the potential is smooth enough, which is not the case in our setting. For non smooth potentials, Kucherenko and Quas can only deal with two-sided subshifts. Thus our result seems of a different nature.
Another result on freezing phase transition can also be found in \cite{Kucherenko-Thompson-20} (Proposition 3), or in \cite{BKL-23} for a theoritical result.

{\bf Acknowledgements}: The authors wish to thank Jean-René Chazottes for useful discussions on a previous version of the paper.

\subsection{Background}
For subshifts and substitutions, we refer to \cite{Durand-Perrin-22} and \cite{Queffelec}.

Let $\mathcal{A}$ be a finite set called the alphabet with cardinality $D\ge 2$. Elements of $\CA$ are called \emph{letters} or \emph{digits}.
A non-empty word is a finite  or infinite string of digits. If $u=u_{0}\ldots u_{n-1}$ is a word, a prefix of $u$ is any word $u_{0}\ldots u_{j}$ with $j\le n-1$. A suffix of $u$ is any word of the form $u_{j}\ldots u_{n-1}$ with $0\le j\le n-1$.  
If $v$ is the  finite word $v=v_0\dots v_{n-1}$ then $n$ is called the length of the word $v$ and is denoted by $|v|$. The set of all finite words over $\mathcal{A}$ is denoted by $\mathcal{A}^*$. 

The shift map is the map defined on $\mathcal{A}^\mathbb{N}$ by $\sigma(u)=v$ with $v_n=u_{n+1}$ for every integer $n$.
We endow $\mathcal A$ with the discrete topology and consider the product topology on $\mathcal A^\N$. This topology is compatible with the distance $d$ on $\mathcal A^{\mathbb N}$ defined by
$$d(x,y)=\frac{1}{2^n}\quad \text{if}\quad n=\min\{i\geq 0,  x_i\neq y_i\}.$$

A substitution $H$ is a map from an alphabet $\mathcal{A}$ to the set $\mathcal{A}^*\setminus\{\epsilon\}$ of nonempty finite words on $\mathcal{A}$. It extends to a morphism of $\mathcal{A}^*$ by concatenation, that is $H(uv)=H(u)H(v)$.

Several basic notions on substitutions are recalled in Section \ref{sec-substi}. We also refer to  \cite{Pyth.02}. We only recall here the notions we need to state our results.

\medskip
Let $H$ be a substitution over the alphabet $\mathcal{A}$, the subshift associated to $H$ is a subset $\K_H$ of $\mathcal A^{\mathbb N}$ such that $x\in\K_H$ if  and only if for every non-negative integers $i, j$ the word $x_i\dots x_{j+i}$ appears in some $H^n(a)$ for a letter $a$. It is called the {\bf subshift} associated to the substitution. 

A subshift $X$ is said to be {\bf minimal} if every orbit under the shift of an element of $X$ is dense in $X$. An {\bf invariant measure} is a probability measure $\mu$ on $X$ such that for every measurable set $A$, we have $\mu(\sigma^{-1}A)=\mu(A)$. A subshift is said to be {\bf uniquely ergodic} it it has only one invariant measure.

In all the following we will restrict to the following example:
\begin{definition}
The Thue Morse substitution is defined on the alphabet $\{0,1\}$ by 
$$\theta:\begin{cases}0\mapsto 01\\ 1\mapsto 10\end{cases}$$
\end{definition}

In all the following we will denote $\mathrm{TM}$ the subshift associated to $\theta$. Moreover $\mathrm{TM}$ is also the orbit closure of a fixed point of $\theta$ under the shift action, and $\mathrm{TM}$ is uniquely ergodic with the unique $\s$-invariant probability denoted by $\mu_{\mathrm{TM}}$, see Lemma \ref{lem-cass}. 

Given a subshift $(\K,\s)$ and a potential  $V:\K\to\R$, the pressure function is defined by 
$$P(\be):=\sup\left\{h_{\mu}+\be\int V\,d\mu\right\},$$
where the supremum is taken over all $\s$-invariant probabilities $\mu$ on $X$ and $h_{\mu}$ is the Kolmogorov entropy. An equilibrium state at $\be$ is a measure which realizes the maximum in the previous formula.

\subsection{Main results}

We will consider the following set $\Xi=\{-\varphi: \mathcal A^{\mathbb N}\rightarrow \mathbb R\}$ of functions defined by:
\begin{itemize}
\item $\varphi(x)=0$ if and only if $x$ belongs to $\K,$ 
\item and $\varphi(x)=\frac{\disp g(x)}{n}+o(\frac1n)$ if $d(x,\K)=2^{-n}$ where $g$ is a positive continuous function 
\end{itemize}

\begin{definition}\label{def-potentiel0}
We consider a function $\varphi_0$ such that $\varphi_0(x)=\log{(1+\frac{1}{n+1})}$ if $d(x,\mathbb K)=2^{-n}$. We denote $V_0=-\varphi_0$. Remark that $V_0$ belongs to $\Xi$.
\end{definition}

Then, our main theorems are the following (the definition of the transfer operator, $\mathcal L_{z,\beta,V}$ is given in Equation \eqref{eq-operator}).

\begin{mainth}\label{thm-tm}
If $\mathrm{TM}$ is the subshift associated to the Thue-Morse substitution and $V\in \Xi$, then there exists $\beta_0$ such that for all $\beta>\beta_0, P(\beta)=0$ and the equilibrium measure is the unique invariant measure supported on $\mathrm{TM}$. Moreover we have $\beta_0<16.6$ if $V=V_0$.
\end{mainth}

\begin{mainth}
\label{thm:meta-phase}
Consider the full shift $\mathcal{A}^\mathbb{N}$ and a subshift $\mathbb{K}$ of zero entropy.
 Consider a cylinder $J$ such that $J\cap \mathbb{K}=\emptyset$. Assume there exists $\beta_0>0$ such that for all  $x\in J$,  
$\disp \mathcal{L}_{0,\beta_{0},V_0}(\BBone_J)(x)<1$ holds. 
Then for every $\be\ge\be_{0}$, every equilibrium measure for the potential $V_0$ gives zero measure to $J$. 

Morever if the subshift is minimal and uniquely ergodic, then $\beta_0$ is independent of $J$, then $\mu_{\mathbb{K}}$ is the unique equilibrium state and thus $P(\beta)=0$ for $\beta\geq \beta_0$.

\end{mainth}

\section{More definitions and tools}\label{sec-substi}

\subsection{Words, languages and special words}

For this paragraph we refer to \cite{Pyth.02}.

\begin{definition}
A word $v=v_0\dots v_{r-1}$ is said to occur at position $m$ in an infinite word $u$ if for all $i\in[0;r-1]$ we have $u_{m+i}=v_i$. We say that the word $v$ is a \emph{factor} of $u$. 

For an infinite word $u$, the language of $u$ (respectively the language of length $n$ of $u$)  is the set of all words (respectively all words of length $n$) in $\mathcal{A}^*$ which occur in $u$. We denote it by $L(u)$ (respectively $L_{n}(u)$).

\end{definition}

\begin{definition}\label{def-rec-lin} 
An infinite word $u$ is said to be recurrent if every factor of $u$ occurs infinitely often in $u$. 

\end{definition}

Remark that $u$ is recurrent is equivalent to the fact that $\sigma$ is onto on the adherence, denoted $\K_u$, of the orbit of $u$. Moreover we have equivalence between $\omega\in \mathbb K_u$ and $L(\omega)\subseteq L(u)$. Thus the language of the adherence of the orbit of $u$ is equal to the language of $u$. 
Remark that it is the case for the Thue Morse substitution, where $\mathrm{TM}=\mathrm{TM}_u$ if $u$ is a fixed point, see \cite{Pyth.02}.

For a substitution $H$, the language of the substitution $H$ is the set of words which are factors of some $H^n(a), a\in\mathcal{A}, n\in\mathbb N$. It will be denoted by $L_{H}$. Except in very special case, it is equal to the set of words which are factors of elements of $\K_H$.

A language is said to be factorial if for every word in the language all its factors are also inside the language. The language is also said to be extendable if every word in the language has a left (and right) extension which is also in the language. First we recall well-known definitions concerning combinatorics of words \cite{Cass.94}. 
\begin{definition}\label{def-bispe}
Let $L=(L_n)_{n\in\mathbb{N}}$ be a factorial and extendable language. 

 For $v \in L_n$ let us define the two quantities
$$
m_{l}(v)= card\{a\in \mathcal{A}, av\in L_{n+1}\},  m_{r}(v)= card\{b\in \mathcal{A},vb\in L_{n+1}\},$$
and the two others
$$m_{b}(v)= card\{(a,b)\in \mathcal{A}^2, avb\in L_{n+2}\},
i(v) = m_b(v)-m_r(v)-m_l(v)+1.$$

A word $v$ is called right special if $m_{r}(v)\geq 2$. It is called left special if $m_{l}(v)\geq 2$. A word $v$ is called bispecial if it is right and left special. 

Finally for a bispecial word $v$, a word $avb$ which is not in the language, and such that $av, vb$ are in the language is called a forbidden extension of $v$.

\end{definition}

\begin{definition}
A  word $v$ such that $i(v)< 0$ is called a weak bispecial. 
A word $v$ such that $i(v)>0$ is called a strong bispecial.
A bispecial word $v$ such that $i(v)=0$ is called a neutral bispecial.
\end{definition}
For a two-letter alphabet, a bispecial word can only fulfill $i(v)=-1$ (weak bispecial), $i(v)=0$ (neutral bispecial) or $i(v)=1$ (strong bispecial).

The complexity function of a language is a map $p:\mathbb N\rightarrow \mathbb N$ such that 
$$p(n)=card (L_n).$$

\begin{definition}
Let $H$ be a substitution. 
We say that the word $u\in L_H$ is {\bf uniquely desubstituable} if there exists only one triple $(s,v,p)$ such that $u=sH(v)p$ 
where 
\begin{enumerate}
\item $p$ is a proper prefix of $H(\wh p)$ and $\wh p\in \mathcal A$,
\item $s$ is a proper suffix of $H(\wh s)$ and $\wh s\in \mathcal A$,
\item $\wh s v\wh p$ is a word in $L_{H}$.
\end{enumerate}
 \end{definition}
 
 \subsection{Background on Thue-Morse substitution}
We will need the following result, see \cite{Pyth.02} or \cite{Durand-Perrin-22} for example. Since we are on a two-letter alphabet, if $a$ is a letter we denote $\overline{a}$ the other letter of the alphabet.

 \begin{theorem}
The subshift associated to the Thue-Morse substitution $\theta$ is uniquely ergodic, minimal and thus recurrent. 

The complexity function of the language of the Thue-Morse substitution fulfills $p(n)\leq 4n$ for $n\geq 1$. 
\end{theorem}

 \begin{lemma}\label{lem-cass}
The Thue-Morse substitution and its language $L_{\theta}$ fulfill:
\begin{itemize}
\item The fixed point which begins with $0$ can be written $$u=0110100110010110100101\dots$$

\item The non-uniquely desubstituable words of $L_{TM}$ are $0,1,01,10,010,101$.
\item Every word of length at least $4$ in $L_{TM}$ is uniquely desubstituable.
\item The neutral bispecial factors are $0$ and $1$. The weak bispecial factors are $\theta^i(010)$ and $\theta^i(101)$ for some $i\geq 0$.
\item The forbidden extensions of a bispecial word are (up to exchange $0-1$) $000$ and words like $a\theta^k(010)b$ with $\begin{cases}a=b\in \mathcal{A}, \quad k=2l\\ a=\overline{b}\quad k=2l+1\end{cases}$.

\end{itemize}
\end{lemma}

\begin{definition}\label{def-generation}
A bispecial word of $L_{TM}$ is said to be of generation zero if it belongs to $\mathcal E=\{0, 1, 010, 101\}$.
A bispecial word is said to be of generation $i\geq 1$ if it is equal to $\theta^i(010)$ or $\theta^i(101)$. 
\end{definition}

\subsection{Accidents}\label{sec:acc}
Consider a substitution $H$ and let $\mathbb K_H$ be the subshift associated to it.
Let $x$ be an element of $\mathcal{A}^\mathbb N$ which does not belong to $\mathbb K_H$. The word $w$ is the maximal prefix of $x$ such that $w$ belongs to the language of $\K_H$. 
Thus we obtain $d(x,\K_H)=2^{-d}$ with $x=w\dots$ and $w=x_0\dots x_{d-1}$. Let us denote $\delta(x)=d$, {\it i.e} $\delta(x)$ is the length of the longest prefix of $x$ in $L_H$. 

Remark that, for the substitution $\theta$, the word $w$ is non-empty since every letter is in the language of $\mathbb K$. Then, $w$ is the unique word such that 
$$x=wx', w\in L_{H}, wx'_{0}\notin L_{\theta}.$$

For a fixed $x\notin \mathbb K_H$, the accident times are ordered which allows to define the notion of $j^{th}$ accident with $j\ge 1$. This is done more formally in Definition  \ref{def:long-bisp-acc}.

\begin{definition}\label{def:long-bisp-acc}
We define
\begin{align*}
b_1&=b=\min\{j\geq 1, d(\sigma^jx,\mathbb K)\leq d(\sigma^{j-1}x,\K_H)\}\\
b_2&=\min\{j\geq 1, d(\sigma^{j+b_1}x,\mathbb K_H)\leq d(\sigma^{j+b_1-1}x,\K_H)\}\\
b_3&=\min\{j\geq 1, d(\sigma^{j+b_1+b_2}x,\mathbb K_H)\leq d(\sigma^{j+b_1+b_2-1}x,\K_H)\}\\
\dots
\end{align*}
Set $B_0=0, B_j=b_1+\dots +b_j$.
Then, the integer $B_j, j\geq 0$ is the {\bf $j^{th}$ accident time for $x$} and $d_{j}:=\delta(\sigma^{B_{j}}x)$ is its depth.  
The word $x_{B_j}\dots x_{d_{j-1}-1}$ is called  the {\bf $j^{th}$ accident-word for $x$}. Its length is called the {\bf length of the $j^{th}$ accident for $x$}.
\end{definition}

\begin{remark}
\label{rem-accident0}
By convention, the $0^{th}$ accident is at time zero. 
$\blacksquare$\end{remark}

Figure \ref{fig-accidents} illustrates the next lemma which appears in \cite{Bruin-Leplaid-13}. 
\begin{lemma}\label{lem:accident-bispecial}
Assume $\mathcal{A}$ is of cardinal $2$, and consider a substitution $H$ on this alphabet. Let $x$ be an infinite word not in $\mathbb K_H$.
Assume that $\delta(x)=d$ and that the first accident appears at time $0<b\leq d$, then the word $x_b\dots x_{d-1}$ is a non strong bispecial word of $L_{H}$. It is called the first accident-word. 
\end{lemma}
\begin{proof}
By definition of accident, we have $\delta(\sigma^bx)>d-b$, thus the word $x_b\dots x_{d-1}x_d$ belongs to $L_{H}$. Moreover $x_0\dots x_{d-1}$ belongs to $L_H$ and $x_0\dots x_d$ does not. Thus $x_0\dots x_{d-1}$ has a right extension in $L_{H}$ which is different from $x_0\dots x_d$. We conclude that the word $x_b\dots x_{d-1}$ has two right extensions in $L_{H}$ (one which belongs to the language of $x$ and one in $L_H$). Moreover we can prove that this word has also two left extensions in $L_{H}$ by definition of $b$ as the time of first accident. 
Thus we conclude that $x_b\dots x_{d-1}$ is a bispecial word. The same argument shows that it is not a strong bispecial word.
\end{proof}

\begin{remark}
\label{rem-acci-left-rigthspe} On a two-letter alphabet, we have to be more careful: the word $x_0\dots x_{d-1}$ is not right special in the language of $\K_H$. Moreover, and again if $\CA$ has cardinality two, if $x=\s(z)$ and there is an accident at time 1 for $z$, then $x_0\dots x_{d-1}$ is not left-special.
$\blacksquare$\end{remark}

\begin{center}
\begin{figure}[htbp]
\begin{tikzpicture}
\draw[dashed](0,1)--(2,0);
\draw(0,0)--(10,0);
\draw (3,-1) node[above]{$w$};
\draw[<->](0,-1)--(6,-1);
\draw[dashed] (-2,1)--(0,0);
\draw (-1.5,1) node{$y$};
\draw[dashed] (6,0)--(8,-1);
\draw (0.8,1) node{$y'$};
\draw (2,-0.1)--(6,-.1);
\draw[dashed] (8,0)--(10,-1);
\draw (8.2,-1) node{$y$};
\draw (10.2,-1) node{$y'$};
\draw (10.5,0) node{$x$};
\draw (6,0) node[above]{$d_0$};
\draw (2,0) node[above]{$b_1$};
\draw (8,0) node[above]{$d_1$};
\draw (6,1)--(10,1);
\draw (7,1) node[above]{$x'$};
\end{tikzpicture}
\caption{Accidents-Dashed lines indicate infinite words in $\mathbb K$.}\label{fig-accidents}
\end{figure}
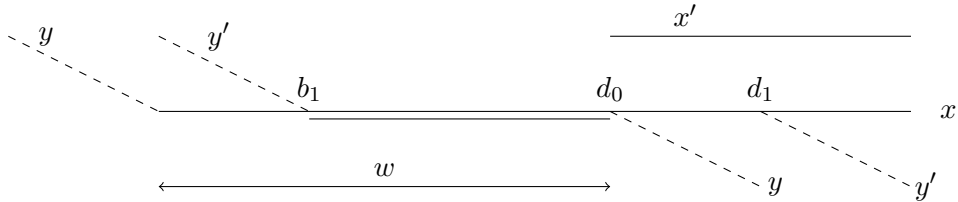
\end{center}

Remark that it could happen that the second accident appears before the first one has finished: It means the second accident-word overlaps with the first one.

\begin{lemma}\label{lem:deux-acc}
Consider $x$ such that $\delta(x)=d$. Denote by $B_i, B_{i+1}$ the times of two consecutive accidents. Assume the two bispecial words defined by the accidents do not overlap, then we have:
$$\delta(\sigma^ix)=\begin{cases}
d-i, 0\leq i< B_i\\
d_1-B_1-i, B_i\le i< B_{i+1}
\end{cases}$$
\end{lemma}
\begin{proof}
It is a simple application of the definition of accident. See also Figure \ref{fig-accidents} with $B_{1}=b$.
\end{proof}


\section{Main tool of the proof of Theorem \ref{thm-tm} }
\subsection{Scheme of the proof}
The Thermodynamic Formalism was introduced in dynamical systems by Sinai, Ruelle and Bowen (\cite{sinai,Ruelle.78,Bow.08}). The main tool for  a uniformly hyperbolic system $(X,\sigma)$ and a H\"older continuous potential $V$ is the transfer operator:
$$\CT(g)(x):=\sum_{y\in\s^{-1}(x)}e^{V(y)}g(y),$$
for $x\in X$ and $g$ a continuous function on $X$.

Hyperbolicity and H\"older continuity combine themselves to give nice spectral properties to this operator. The main point is that  the pressure for $V$ is the logarithm of the spectral radius for $\CT$ which has a single dominating eigenvalue. 

For systems with weaker hyperbolicity or potentials with weaker regularity, it may be harder to get the same spectral properties. A way to recover them is to consider an inducing scheme. Several methods exist in the literature. We shall use here the one summarized in \cite{lepl-survey}. 

This result will be stated in Theorem \ref{thm:meta-phase}, for which we refer to the appendix for a proof, since this result has no complete proof in the litterature.

We consider $J$ a cylinder outside $\mathbb K$ defined by a word $w_J\notin L_{\mathbb K}$. Consider the first return map $f$ to $J$ of $\sigma$, with return time{\color{orange}\footnote{not defined everywhere.}} $\tau(x)=\min\{n\geq 1, \sigma^n(x)\in J\}$.


Then we define, for each $\beta>0$ and $z\in\mathbb R$, an induced {transfer operator} by:

\begin{equation}\label{eq-operator}
\mathcal{L}_{z,\beta,V}(g)(x)=\displaystyle\sum_{n\in\mathbb{N}}\sum_{\substack{y\in J\\ \tau(y)=n\\ \sigma^n(y)=x}}e^{\beta(S_nV)(y)-nz}g(y),
\end{equation}
where $(S_nV)(y)=\displaystyle\sum_{k=0}^{n-1}V\circ \sigma^k(y)$ and $g$ is a continuous function from $J$ to $\R$.

\begin{remark}
The function $\CL_{Z,\beta,V}(g)$ is continuous if $g$ is continuous. Thus the operator is defined on the set of continuous functions from $J$ to  $\mathbb R$.
\end{remark}

We will compute  $\mathcal{L}_{0,\beta,V_0}(\BBone_J)(x)$ for the particular potential $V_0=-\varphi_0$, see Definition \ref{def-potentiel0}, and deduce the result for this potential. With Lemma \ref{fin} we will deduce the result for every potential in $\Xi$.

With the help of Theorem \ref{thm:meta-phase}, our strategy of proof is the following: the rest of the proof consists in showing that $\CL_{0,\be,V_0}(\BBone_J)(x)$
is strictly smaller than $1$ for $\beta$ large enough and independent of $J$, which will be done in Proposition \ref{prop:final} and the following ones.  

\subsection{Return words and minimal forbidden words}
Let us recall that $w_J$ is the word which defines the cylinder $J$. For $x\in J$ we have to compute, for $\beta$ large enough: 
 
\begin{equation}\label{eq-operator2}
\mathcal{L}_{0,\beta,V_0}(\BBone_J)(x)=\displaystyle\sum_{n\in\mathbb{N}}\sum_{\substack{y \in J\\ \tau(y)=n\\ \sigma^n(y)=x}}e^{-\beta(S_n\varphi_0)(y)}.
\end{equation}

Remark that this sum is infinite if $\beta$ is small enough, see \cite{Ishaq-Leplaideur-22}. All the computations will be made in $\overline{\mathbb R_+}$. It will allow us to replace 'diverges in the positive direction' by 'converges to infinity'. Note that such a point $y$ is of the form $y=ux$, where $u=u_0\dots u_{n-1}$ and $uw_J$ has $w_J$ for prefix.
   
   Now, due to the form of our potential, we claim that $S_{n}(\varphi)(y)$ does only depend on $u$, and thus $\mathcal{L}_{0,\beta,V_0}(\BBone_J)(x)$ is also independent of $x$ thus constant.

\begin{definition}
Let $w$ be a word, then we consider $R(w)=\{u\neq\varepsilon, uw\in w\mathcal A^*, uw\notin\mathcal{A}^{+}w\mathcal{A}^{+}
\}$. This is called the set of return words of $w$ (in $\mathcal A^{\mathbb N}$).
\end{definition}

A {\bf minimal forbidden word of $\mathbb K$} is a word $w$ which is not in $L_{\mathrm{TM}}$, and has minimal length in the sense that each of its proper factors is in $L_{\mathrm{TM}}$.

\begin{remark}
Lemma \ref{lem-indpdt-cylindre} will explain why we use only minimal forbidden words. 
\end{remark}

\begin{lemma}\label{lem-return-bisp}
The word $w$ is a minimal forbidden word for $L_\mathrm{TM}$ if and only if it is a forbidden bilateral extension of a bispecial word.
\end{lemma}
\begin{proof}
Assume $w$ is a minimal forbidden word. Let us write $w=a'ua$ with $u$ a word of length at least one, since $\mathbb L_{\mathrm{TM}}$ contains all words of length two. Then $a'u$ is inside the language $L_{\mathrm{TM}}$ by definition, thus there exists a letter $b$, such that $a'ub$ is also inside the language. The letter $b$ is different from $a$, otherwise $w$ would be in $L_{\mathbb K}$. By symmetry, there exists $c\neq a'$ such that $cua$ is also in the language. Thus $u$ is right special and $u$ is left special, thus $u$ is a bispecial word of $L_\mathrm{TM}$.
Conversely, if $w$ is a forbidden extension, then it is a forbidden word. It is clearly a minimal forbidden word.

\end{proof}

\begin{proposition}\label{prop-operateur-prod-mot}
With the previous notations we obtain
$$\mathcal{L}_{0,\beta,V_0}(\BBone_J)(x)= \sum_{u\in R(w_J)}\prod_{k=0}^{|u|-1}(1+\frac{1}{\delta(\sigma^k(uw_J))})^{-\beta}.$$
\end{proposition}
\begin{proof}
Recall Equation \eqref{eq-operator}:
$$\mathcal{L}_{0,\beta,V_0}(\BBone_J)(x)=\displaystyle\sum_{n\in\mathbb{N}}\sum_{\substack{y\in J\\ \tau(y)=n\\ \sigma^n(y)=x}}e^{-\beta (S_n\varphi)(y)}.$$

Such an infinite word $y$ can be written as $y=ux$ where $u$ is a word of length $n$ which belongs to $R(w_J)$:

$$\mathcal{L}_{0,\beta,V_0}(\BBone_J)(x)=\displaystyle\sum_{n\in\mathbb{N}^*}\sum_{\substack{u\in R(w_J)\\ |u|=n}}e^{-\beta(S_n\varphi)(ux)}.$$

$$(S_n\varphi)(ux)=\displaystyle\sum_{k=0}^{n-1}\varphi\circ \sigma^k(y)=\displaystyle\sum_{k=0}^{n-1}\log(1+\frac{1}{\delta(\sigma^k(y))})=\log\prod_{k=0}^{n-1}(1+\frac{1}{\delta(\sigma^k(y))}),$$
$$e^{-\beta S_n(\varphi)(y)}=\prod_{k=0}^{n-1}(1+\frac{1}{\delta(\sigma^k(y))})^{-\beta}.$$

Moreover we have $\delta(\sigma^k(y))=\delta(\sigma^k(uw_J))$, thus we can conclude. 
\end{proof}

\subsection{Sum}
Assume that $w_J$ is a minimal forbidden word of $L_{\mathrm{TM}}$ which defines $J$.

In order to describe the sum of Proposition \ref{prop-operateur-prod-mot} we introduce, for a fixed word $u$ with exactly $M\geq 0$ accidents, the following notations: the accident words are denoted $v^1,\dots, v^M$ and the maximal words which are in the language are denoted $u^0,\dots, u^M$ where $u^i$ has $v^{i+1}$ for suffix, see Figure \ref{fig:suite-accidents}.

\begin{figure}
\begin{tikzpicture}[scale=1]
\draw (0,-.4)--(6,-.4);
\draw(6.2,-.4) node{$u$};
\draw (0,-.8)--(2,-.8);
\draw(2.2,-.8) node{$u^0$};
\draw (1.5,-1.2)--(2,-1.2);
\draw(2.2,-1.2) node{$v^1$};
\draw (1.5,-1.5)--(4,-1.5);
\draw(4.2,-1.5) node{$u^1$};
\draw (3.6,-1.9)--(4,-1.9);
\draw(4.2,-1.9) node{$v^2$};
\draw (3.6,-2.3)--(5,-2.3);
\draw(5.2,-2.3) node{$u^2$};
\end{tikzpicture}
\caption{Global picture of the accidents in a word $u$: the first bispecial word $v_1$, then the second accident with extension $a_1v_1b_1$, and so on.}\label{fig:suite-accidents}
\end{figure}
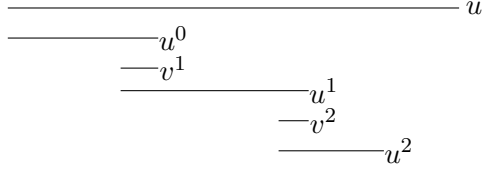

\begin{proposition}\label{prop-fract-bisp}
With previous notations, we have
$$\sum_{u\in R(w_J)}\prod_{k=0}^{|u|-1}(1+\frac{1}{\delta(\sigma^k(uw_J))})^{-\beta}=\sum_{M\geq 0}
\displaystyle\sum_{\substack{ u\in R(w_J)\\\text{M accidents}}} [\frac{(|u^0|+1)\dots (|u^{M-1}|+1)(|u^M|+1)}{(|v^1|+1)\dots (|v^M|+1)(|w_J|-1)}]^{-\beta}.$$
\end{proposition}

\begin{proof}
We make a partition of the set of words $u$ with the number $M$ of accidents

$$\sum_{u\in R(w_J)}\prod_{k=0}^{|u|-1}(1+\frac{1}{\delta(\sigma^k(uw_J))})^{-\beta}=\displaystyle\sum_{M\geq 0}\sum_{\substack{u\in R(w_J)\\ \text{M accidents}}}\prod_{k=0}^{|u|-1}(1+\frac{1}{\delta(\sigma^k(uw_J))})^{-\beta}.$$

Between two accidents the product can be simplified :
If $\delta(uw_J)=p$ and the first accident-word is a prefix of $\sigma^d(uw_J)$, we obtain the following formula where $p-d$ is the length of the accident word.
$$(1+\frac{1}{p})(1+\frac{1}{p-1}) \dots (1+\frac{1}{p-d+1})=\frac{p+1}{p-d+1}.$$

We will deduce the result as follows:
we apply this formula with words $u^0,\dots, u^M$ and accident-words $v^1, \dots, v^{M+1}$. And we remark that the product $\prod_{k=0}^{n-1}$ is equal to the product $\prod_{k=1}^n$ where $u^M$ is the maximal prefix which ends with $w_J$ minus the last letter. 

\end{proof}

\begin{definition}\label{def-SM}
Let  
$$S_M(w_J)=\displaystyle\sum_{\substack{u\in R(w_J),\\ \text{M accidents}}} [\frac{(|u^0|+1)\dots (|u^M|+1)}{(|v^1|+1)\dots (|v^M|+1)(|w_J|-1)}]^{-\beta}.$$
\end{definition}

Note that $S_M(w_J)\in[0,+\infty]$. Remark that the previous proposition gives 
\begin{equation}\label{eq-l-sum}
\sum_{u\in R(w_J)}\prod_{k=0}^{|u|-1}(1+\frac{1}{\delta(\sigma^k(uw_J))})^{-\beta}\leq \displaystyle\sum_{M\geq 0}S_M(w_J).
\end{equation}

Now the goal is to prove

\begin{proposition}\label{prop:final}
If $\beta>4$, then $S_M(w_J)$ is finite for all $M\geq 0$. 
If $\beta>17$, then $\sum_{M} S_M(w_J)$ converges and $\sum_{M\geq 0} S_M(w_J)$ is less than $1$.
\end{proposition}

The value of $\beta$ is not optimal, see Section \ref{sec-algo} for the expected value.

\section{Technical lemmas}
\subsection{Infinite matrices}

\begin{definition}
In the following we will consider some array $A$ in $I\times I$, where $I$ is a countable set,  with values in $\overline{\mathbb R_+}$, and some element $P$ of $\mathbb R_+^{I}$. The product $AP$ is well defined in $\overline{\mathbb R_+}^{I}$, and we will call $A$ an infinite matrix and $P$ a vector. All the computations will be made in $\overline{\mathbb R_+}$ in order to avoid problems of convergence. We will denote $E$ the vector with all coordinates equal to $1$.
\end{definition}

We introduce
$I=\{(a,v,b),a,b\in\mathcal A, avb\quad \text{minimal forbidden word}\}$ and

\begin{definition}
Let us introduce the infinite matrix $A$ with coefficients
$$A_{(a_1,v^1,b_1), (a_2,v^2,b_2)}=\sum_{\substack{u\in L_{\mathrm{TM}}\\ u\in v^1b_1\mathcal{A}^*\\ u\in \mathcal{A}^*a_2v^2}} \frac{(|v^1|+1)^{\beta/2}(|v^2|+1)^{\beta/2}}{(|u|+1)^\beta}.$$
\end{definition}

\begin{lemma}\label{lem-maj-Sm}
Consider $w_J$ the minimal forbidden word which defines $J$, then $w_J=avb$ where $v$ is a bispecial word and we obtain
$$\sum_{M\geq 0}S_M(w_J)=\sum_{M\geq 0}(A^{M+1})_{(a,v,b),(a,v,b)}.$$
\end{lemma}
\begin{proof}
Let $u$ be a return word of $w_J$ with $M$ accident words $v^i, 1\leq i\leq M$. Each of them has a forbidden extension $a_iv^ib_i\in I$. Set $v^0=v^{M+1}=v$  and let $u^i, 0\leq i\leq M$ be words between accident-words such that $|u^i|>|v^{i+1}|$ and $|u^i|>|v^i|$, see Figure \ref{fig:suite-accidents}.

We use Proposition \ref{prop-fract-bisp} and obtain
$$S_M(w_J)=\prod_{i=0}^M(\frac{|u^i|+1}{|v^{i+1}|+1})^{-\beta}=\prod_{i=0}^M(\frac{|v^{i+1}|+1}{|u^i|+1})^{\beta}.$$
We deduce since $v_{M+1}=v_0$ the following expression
$$S_M=\prod_{i=0}^M \frac{(|v^i|+1)^{\beta/2}(|v^{i+1}|+1)^{\beta/2}}{(|u^i|+1)^\beta}.$$

We obtain with Proposition \ref{prop-fract-bisp}

$$\sum_{M\geq 0} S_M(w_J)=\sum_{M\geq 0}\displaystyle\sum_{(a_i,v^i,b_i)_{1\leq i\leq M}\in I^M}\sum_{\substack{u_0\dots u_M\in L_{\mathrm{TM}}\\ u^i\in v_ib_i\mathcal{A}^*\\ u^i\in \mathcal{A}^*a_{i+1}v^{i+1}}}\prod_{i=0}^M\frac{(|v^i|+1)^{\beta/2}(|v^{i+1}|+1)^{\beta/2}}{(|u^i|+1)^\beta}.
$$

$$=\sum_{M\geq 0}\sum_{(a_i,v^i,b_i)_{1\leq i\leq M}\in I^M}\prod_{i=0}^MA_{(a_i,v^i,b_i),(a_{i+1},v^{i+1},b_{i+1})}=\sum_{M\geq 0} (A^{M+1})_{(a,v,b), (a,v,b)}$$
\end{proof}

Now we estimate the coefficients of $A$.

By definition the accident words are non-strong bispecial words. Thus by property of the Thue-Morse substitution we can describe all the accident words as $v=\theta^{i}(v^0)$ with $v^0\in\{010,101\}$ if $i>0$ and if $i=0, v^0\in\{0,1, 010, 101\}$.

\begin{definition}
For simplicity, we will denote 
$$A_{i,j}=\sum_{\substack{(a_1,v^1,b_1) \text{ generation i}\\(a_2,v^2,b_2) \text{ generation j} }}A_{(a_1,v^1,b_1), (a_2,v^2,b_2)}.$$
\end{definition}

Now we estimate the coefficients of $A$ with the next lemmas. Remark that we lose a fixed factor by looking at $A_{i,j}$ but it will not change the result :

\begin{lemma}
If $i>0, j>0$ then $\begin{cases} A_{i,j}\leq A_{i-j,0}\quad  \text{if}\quad i\geq j\\ A_{i,j}\leq A_{0,j-i} \quad \text{otherwise}\end{cases}$.
\end{lemma}
\begin{proof}
Consider $u\in v^1b_1\mathcal{A}^*$ and $u\in \mathcal{A}^*a_2v^2$. As $v^1, v^2$ are not of generation $0$ we have for $k=1,2$, that $v_k=\theta(v'_k)$, then by  point 3 of Lemma \ref{lem-cass} there exists a unique $u'$ such that $u=\theta(u')$   with $u'\in L_\theta, u'\in v'_1b_1\mathcal{A}^*, u'\in \mathcal{A}^* \overline{a_2}v'_2$.

$$A_{(a_1,v^1,b_1), (a_2,v^2,b_2)}=\sum_{\substack{u\in L_{\mathrm{TM}}\\ u\in v^1b_1\mathcal{A}^*\\ u\in \mathcal{A}^*a_2v^2}} \frac{(|v^1|+1)^{\beta/2}(|v^2|+1)^{\beta/2}}{(|u|+1)^\beta}.$$

We deduce with $|u|= 2|u'|$
$$ A_{(a_1,v^1,b_1), (a_2,v^2,b_2)}=\sum_{u'}\frac{(2|v'_1|+1)^{\beta/2}(2|v'_2|+1)^{\beta/2}}{(2|u'|+1)^\beta}.$$
Remark that if $0<x<y$ we have the following inequality

\begin{equation}\label{inegalite}
\frac{x+1/2}{y+1/2}<\frac{x+1}{y+1}.
\end{equation} 

Thus we deduce
$$A_{(a_1,v^1,b_1), (a_2,v^2,b_2)}\leq \sum_{u'}\frac{(|v'_1|+1)^{\beta/2}(|v'_2|+1)^{\beta/2}}{(|u'|+1)^\beta}=A_{(\overline{a_1},v'_1,b_1), (\overline{a_2},v'_2,b_2)}.$$
We apply the same process of desubstitution several times and conclude by induction on $min(i,j)$ and summation in order to obtain $A_{i,j}$.
\end{proof}

Before next lemmas, we recall a very classical result used in the following.
\begin{lemma}\label{lem-integrale}
If $\be>2$, and $n_0\geq 2$ is an integer, then 
$$\sum_{n\geq n_0} \frac{1}{n^{\beta-1}}\leq \frac{1}{(\beta-2)(n_0-1)^{\beta-2}}.$$
\end{lemma}

\begin{lemma}\label{lem-calcul1}
There exists a real function $\varepsilon_1:[0,+\infty)\rightarrow \R$  with $\lim_{+\infty} \varepsilon_1(\beta)=0$, such that if $i>0$, then for $\beta>2$ we have 
$$A_{i,0}\leq (\frac{3}{2^{\beta/2}})^{i}(\frac{4}{5})^{\beta/2}(4+\varepsilon_1(\beta)).$$

\end{lemma}
\begin{proof}
Let $(a_1,v^1,b_1)$ be minimal forbidden words of generation $i>0$ and $(a_2, v^2, b_2)$ of generation $0$.
$$\sum_{\substack{(a_2,v^2,b_2)\\ \text{generation} 0}} A_{(a_1,v^1,b_1),(a_2,v^2,b_2)}\leq \sum_{\substack{u\in L_{\mathrm{TM}},\\ u\in v^1b_1A^*}}\frac{(|v^1|+1)^{\beta/2}4^{\beta/2}}{(|u|+1)^{\beta/2}(|u|)^{\beta/2}}$$
Since $i>0$ we desubstitute $v^1$ so that $v_1=\theta(v'_1)$, and we obtain either $|u|=2|u'|$ or $|u|=2|u'|+1$ with $u=\theta(u')$ or $u=\theta(u')p$ with $p\in\mathcal A$ and $u'\in v'_1 b_1\mathcal{A}^*$. Thus we always have $|u|\geq 2|u'|$, and
there are $3$ possibilities to desubstitute $u$, and we obtain
$$\sum_{(a_2,v^2,b_2)} A_{(a_1,v^1,b_1),(a_2,v^2,b_2)}\leq 3\sum_{\substack{u'\in L_{\mathrm{TM}},\\ u\in v'^1b_1A^*}}\frac{(2|v'^1|+1)^{\beta/2}4^{\beta/2}}{(2|u'|+1)^{\beta/2}(2|u'|)^{\beta/2}}$$
With inequality \eqref{inegalite}, since $|v'^1|<|u'|$ we deduce

$$\leq \frac{3}{2^{\beta/2}}\sum_{u'\in L_{\mathrm{TM}}} \frac{(|v'^1|+1)^{\beta/2}4^{\beta/2}}{(|u'|+1)^{\beta/2}(|u'|)^{\beta/2}} $$
We iterate the process $i$ times and obtain the following,
 $$A_{i,0}=\sum_{\substack{(a_1, v^{(i),1},b_1)\\ \text{generation i}}} \frac{3}{2^{\beta/2}}*\frac{(|v'^1|+1)^{\beta/2}4^{\beta/2}}{(|u'|+1)^{\beta/2}(|u'|)^{\beta/2}}$$
 $$A_{i,0}\leq \sum_{\substack{(a'_1,v'^{(i),1},b'_1)\\ \text{generation}\quad 0,\\ |v'^1|=3}}(\frac{3}{2^{\beta/2}})^i\sum_{\substack{u'\in L_{\mathrm{TM}},\\ u'\in v'^{(i),1}b_1A^*}} \frac{4^\beta}{(|u'|+1)^{\beta/2}(|u'|)^{\beta/2}}$$
Remark that the map $v'^{(i),1}\mapsto u'$ is injective.
$$A_{i,0}\leq (\frac{3}{2^{\beta/2}})^i\sum_{\substack{u'\in L_{\mathrm{TM}},\\ |u'|\geq 4}} \frac{4^\beta}{(|u'|+1)^{\beta/2}(|u'|)^{\beta/2}}
$$
There are $4$ words of length four with prefixes $010$ or $101$, and for a fixed length $n\geq 5$ the numbers of words in $L_{\mathrm{TM}}$ is bounded by $4n$ by Lemma \ref{lem-cass}:

$$A_{i,0}\leq (\frac{3}{2^{\beta/2}})^i[4\frac{4^{\beta/2}}{5^{\beta/2}}+\sum_{n\geq 5} \frac{4^\beta 4n}{(n+1)^{\beta/2}n^{\beta/2}}]$$
$$A_{i,0}\leq (\frac{3}{2^{\beta/2}})^i[4. (\frac{4}{5})^{\beta/2}+4^{\beta+1}\sum_{n\geq 5} \frac{1}{n^{\beta-1}}]$$
We use Lemma \ref{lem-integrale}

$$A_{i,0}\leq (\frac{3}{2^{\beta/2}})^i[4. (\frac{4}{5})^{\beta/2}+
20(\frac{4}{5})^{\beta}+
\frac{4*5^2}{\beta-2}(\frac{4}{5})^{\beta}]$$

Since $\frac{2}{\sqrt 5}>4/5$ we 

$$A_{i,0}\leq (\frac{3}{2^{\beta/2}})^i(\frac{4}{5})^{\beta/2}('+\varepsilon_1(\beta))$$

\end{proof}

\begin{lemma}\label{lem-calcul2}
There exists a real function $\varepsilon_2:[0,+\infty)\rightarrow \R$  with $\lim_{+\infty} \varepsilon_2(\beta)=0$, such that for $\beta>2$, $A_{0,0}\leq (20+\varepsilon_2(\beta))(\frac{4}{5})^\beta$.

\end{lemma}
\begin{proof}
$$A_{0,0}=\sum_{\substack{(a_1, v^1,b_1),(a_2, v^2, b_2)\\ gen 0}}\sum_{\substack{u\in L_{\mathrm{TM}},\\ u\in v^1b_1A^*,\\ u\in A^* v^2b_2}}\frac{(|v^1|+1)^{\beta/2}(|v^2|+1)^{\beta/2}}{(|u|+1)^\beta}$$
$$A_{0,0}=\sum_{\substack{u\in L_{\mathrm{TM}},\\ \overline{a_1}u\overline{b_2}\in L_\K\\ a_1ub_2 ***}}\frac{(|v^1|+1)^{\beta/2}(|v^2|+1)^{\beta/2}}{(|u|+1)^\beta}$$
where $**$ means: the word begin and ends by forbidden minimal words of generation zero.

We partition the sum in three:
$$A_{0,0}=\sum_{u, |u|=2}+\sum_{u, |u|=3}+\sum_{u, |u|\geq 4}$$
For every word of length at least $4$ we bound the length of the bispecial word by $3$, and there are two such words. For length $3$ there is no such word $0$, and for length two there are two words. Thus we deduce
$$A_{0,0}\leq 2. \frac{2^\beta}{3^\beta}+0+2\sum_{u\in L_{\mathrm{TM}},[u[\geq 4} \frac{4^\beta}{(|u|+1)^\beta}.$$
For a fixed length $n\geq 0$ the numbers of words in $L_{\mathrm{TM}}$ is bounded by $4n$ by Lemma \ref{lem-cass}:
$$A_{0,0}\leq 2. \frac{2^\beta}{3^\beta}+2.4^{\beta}\sum_{n\geq 4} \frac{4(n+1)}{(n+1)^\beta}\leq 2. \frac{2^\beta}{3^\beta}+2.4^{\beta+1}\sum_{n\geq 5}\frac{1}{n^{\beta-1}}.$$
Now we use Lemma \ref{lem-integrale} and obtain
$A_{0,0}\leq 2. \frac{2^\beta}{3^\beta}+20(\frac{4}{5})^{\beta}+\frac{8.5^2}{(\beta-2)}(\frac{4}{5})^\beta$.

We conclude since $2/3<4/5$ that for $\beta>2$ there exists a function $\varepsilon_2$ such that 
$$A_{0,0}\leq (20+\varepsilon_2)(\frac{4}{5})^\beta.$$

\end{proof}

\begin{lemma}\label{lem_inegalite}
There exists $\lambda(\beta)$ such that $AE\leq \lambda E$ for $\beta>4$ and $\lim_{+\infty} \lambda(\beta)=0$.
\end{lemma}
\begin{proof}For every $i$ we consider $\lambda(\beta)= \sup_i \sum_{j}A_{i,j}$ and obtain with Lemma \ref{lem-calcul1} and Lemma \ref{lem-calcul2}:
$$\sum_j A_{i,j}\leq \sum_{j=0}^{i-1} A_{i-j,0}+A_{0,0}+\sum_{j\geq i+1} A_{j-i,0}\leq A_{0,0}+2\sum_{n\geq 1} A_{n,0}$$
$$\sum_{j} A_{i,j}\leq (20+\varepsilon_2(\beta))(\frac{4}{5})^\beta+2\sum_{n\geq 1}(\frac{3}{2^{\beta/2}})^{n}(\frac{4}{5})^{\beta/2}(1+\varepsilon_1(\beta)).$$
For $\beta>4$, we have $2^{\beta/2}>3$, thus we deduce

$$\sum_{j} A_{i,j}\leq (20+\varepsilon_2(\beta))(\frac{4}{5})^\beta+(\frac{4}{5})^{\beta/2}(1+\varepsilon_1(\beta))\frac{6}{2^{\beta/2}-3}$$

$$\lambda(\beta)=(\frac{4}{5})^{\beta/2}[20(\frac{4}{5})^{\beta/2}+\frac{6}{2^{\beta/2}-3}].$$

\end{proof}

\subsection{Proof of Proposition \ref{prop:final}}
Consider the word $w_J$ which defines $J$, then $w_J$ defines a bispecial $v$ of some generation $i$, and we have $\sum_M S_M(w_J)\leq \sum_M (A^{M+1})_{i,i}$ by Lemma \ref{lem-maj-Sm}. 

By Lemma \ref{lem_inegalite} we obtain $A^{M+1}E\leq \lambda(\beta)^{M+1}E$, then we conclude for all integer $i$ that $A^M_{i,i}\leq \lambda(\beta)^M E_{i,i}=\lambda(\beta)^M$. Thus for $\beta>4$ we deduce that $S_M(w_J)$ is finite.

Now if $\lambda(\beta)<1$ we have $\sum_M S_M(w_J)\leq \frac{\lambda(\beta)}{1-\lambda(\beta)}$. Since $\lim_{\beta\rightarrow+\infty} \lambda(\beta)=0$ we deduce the result. Moreover we deduce that $\sum_M S_M(w_J)<1$ if $\lambda(\beta)<1/2$ which is true for $\beta>16.6$, by numerical computation. Finally remark that the computation is independant of the minimal forbidden word $w_J$, due to the inequality involving $\lambda(\beta)$, thus of $J$.

\section{Proof of Theorem \ref{thm-tm}}

\subsection{Conclusion for the potential $V_0$}
We want to compute $\mathcal{L}_{0,\beta,V_0}(\BBone_J)(x)$. 
We use Proposition \ref{prop-operateur-prod-mot}, then Proposition \ref{prop-fract-bisp} which reduce the problem to the convergence of $\sum S_M(w_J)$. Then
Proposition \ref{prop:final} and Theorem \ref{thm:meta-phase} prove the result for the potential $V_0$.

\begin{lemma}\label{lem-indpdt-cylindre}
The value $\beta_0$ is independent of $J$.
\end{lemma}
\begin{proof}
Consider $J$ a cylinder outside $\mathbb K$. Then $J$ is included in some other cylinder $J'$ defined by a minimal forbidden word
$w$ in $\mathcal L_{\mathbb K}$ and for any $\sigma$ invariant probability measure  we have $\mu(J)\leq \mu(J')$. 

Thus, to prove that an equilibrium measure has support in $\mathrm{TM}$ it is enough to prove that $\mu(J)=0$ for each $J$ defined by a minimal forbidden word.   
We conclude with Theorem \ref{thm:meta-phase} and Proposition \ref{prop:final}.
\end{proof}

\subsection{Last step in the proof of Theorem \ref{thm-tm}}

We conclude the proof with the next lemma:

\begin{lemma}\label{fin}
Assume Theorem \ref{thm-tm} is true for the potential $-\varphi_0$, then it is true for every potential $V(=-\varphi\in \Xi)$.
\end{lemma}
\begin{proof}
If $-V\in \Xi$ then there exists $k,k'>0$ such that $k'\varphi_0\leq-V\leq k\varphi_0$. We deduce that the pressure function of the potential $V$ vanishes for $\beta\geq \frac{\beta_0}{k}$. Since this function is continuous, convex and decreasing there exists $\beta'_c$ such that 
$P(\beta)>0 \quad 0\leq \beta\leq \beta'_c$ and $P(\beta)=0, \beta\geq \beta'_c$. 
The rest of the proof is similar. 

\end{proof}

\section{Algorithm for other substitution}\label{sec-algo}
Consider $F(w)=\displaystyle\sum_{\substack{u\in\mathcal{A}^+,\\ w \;\text{prefix of}\; uw}}\prod_{k=0}^{|u|-1} (1+\frac{1}{\delta(\sigma^k(uw))})^{-\beta}$. 

By Proposition \ref{prop-operateur-prod-mot} in order to prove Theorem \ref{thm:meta-phase}, we need to check if there exists $\beta_0$ such that for all $w$ minimal forbidden word of the language of Thue-Morse substitution, $F(w)<1$.

Let us define $p(x)$ the maximal prefix of $x$ in $L_\theta$. 
Consider $v\in L(\theta)$ and $F_n(v,w)=\displaystyle\sum_{\substack{u\in A^n\\ p(uw)=v}}\prod_{k=0}^{|u|-1} (1+\frac{1}{\delta(\sigma^k(uw))})^{-\beta}$. 
We have $$F(w)=\sum_{n\geq 1} F_n(p(w),w).$$
Moreover we have $F_0(p(w),w)=1$ and 
$$F_{n+1}(v,w)=\sum_{\substack{v'\in L_\theta\\ |v'|\leq n+|w|\\ a\in A, p(av')=v}}F_n(v',w)(1+\frac{1}{|v|})^{-\beta}$$

Now we have a test which decide what is the biggest prefix of $u$ inside $\mathcal{L}_{\mathbb K}$. It is optimised for Thue-Morse language.

\begin{remark}
The algorithm has a cost of $n^ 3$ operations. Indeed the number of prefixes of $u$ of length $n$ inside $\mathcal{L}_{\mathbb K}$ is linear in $n$. The linear recurrence formule is a sum over $n^2$ terms.  

For Thue-Morse substitution, numerical experiments seems to imply $4<\beta<6$. Moreover we can conjecture a behavior for the pressure function $P(\beta)$ like the map $ e^{-n(\beta-4)}+\frac{1}{n^{\beta- 2}}$.
\end{remark}

\section{Appendix: Proof of Theorem \ref{thm:meta-phase}}

To finish we give a complete proof of Theorem \ref{thm:meta-phase}. Part of the proof can be found in the following papers of Leplaideur: \cite{lepl-survey} and \cite{Bruin-Leplaid-13}. We recall
$$\mathcal{L}_{z,\beta,V}(g)(x)=\displaystyle\sum_{n\in\mathbb{N}}\sum_{\substack{\tau(y)=n\\ \sigma^n(y)=x}}e^{\beta.S_n(V)(y)-nz}g(y)$$

\begin{lemma}
There exists $z_c(\beta)$ such that for $z>z_c(\beta)$ the quantity $\mathcal L_{z,\beta, V}(g)(x)$ converges for all $x$ and all $g\in \mathcal C(J,\mathbb R)$.
\end{lemma}

\begin{proof}
For all $y, y'\in J$ with the same return word  we have $S_nV(y)=S_nV(y')$  with $n=\tau(y)=\tau(y')$, since $-\varphi$ only depends on $J$. 

Now we remark that $$\mathcal{L}_{z,\beta, V}(g)(x)=\displaystyle\sum_{n\in\mathbb{N}}\left(\sum_{\substack{\tau(y)=n\\ \sigma^n(y)=x}}e^{\beta(S_nV)(y)}g(y)\right)e^{-nz}$$

$$\mathcal{L}_{z,\beta, V}(g)(x)\leq ||g||_{\infty}\displaystyle\sum_{n\in\mathbb{N}}\left(\sum_{\substack{\tau(y)=n\\ \sigma^n(y)=x}}e^{\beta(S_nV)(y)}\right)e^{-nz}.$$
It is a power serie in $e^{-z}$. Thus it has an abscissa of convergence which does not depend on $x$. 
\end{proof}

Now we explain the link between invariant measures on $(\Sigma,\sigma)$ and invariant measures on $(J,f)$.
If $\mu$ is an invariant measure defined on $\Sigma$ such that $\mu(J)>0$, then we can define an $f$ invariant probability measure $m$ on $J$ by $m(A)=\frac{\mu(A\cap J)}{\mu(J)}$. Conversely, if $m$ is such a measure, then there exists $\mu$ obtained from $m$ if and only if $\int_J \tau dm<\infty$. 

\begin{lemma}\label{lem-zero}
If $\mu$ is an equilibrium measure for $(\Sigma,\sigma, V)$ with pressure $P$ and $\mu(J)>0$, then $m$ is an invariant
measure for $(J,f,S_{\tau}V-\tau P)$ with zero pressure.
\end{lemma}
\begin{proof}
Abramov's formula gives us $h_{m}=\frac{h_\mu}{\mu(J)}$. Moreover we have $\int_J \beta S_\tau Vdm=\frac{1}{\mu(J)} \int_X \beta Vd\mu$. By hypothesis  we deduce 
$$P=h_{\mu}+\int _X\beta V d\mu=\mu(J)[h_m+\int_J \beta S_\tau Vdm]$$
By Kac's lemma we obtain
$$0=\mu(J)[h_m+\int_J \beta S_\tau Vdm-\frac{P}{\mu(J)}]=\mu(J)[h_m+\int_J (\beta S_\tau V-P\tau)dm].$$
$$0=h_m+\int_J (\beta S_\tau V-P\tau)dm$$
We deduce that the pressure of the measure $m$ for the system $(J,f,S_{\tau}V-\tau P)$ is zero. 
\end{proof}

\subsection{Tool from functional analysis}

 We want to use the following theorem by Ionescu- Tulcea, Marinescu. We refer also to  \cite{Hennion-Herve} for more elaborate versions.

\begin{theorem}\cite{Ionescu-Marinescu}\label{ionescu}
Consider a Banach space $X\subset \mathcal C^0(J,\mathbb R)$ with the norm $\|.\|_X$. Consider an operator $\mathcal L$ which acts on $\mathcal C^0(J,\mathbb R)$, and assume
\begin{enumerate}
\item If $(f_n)_n$ is a sequence of functions in $X$ which converges in $\mathcal C^0(J,\mathbb R)$ to a function $f$ and if for all $n\in\mathbb N$, we have 
$\|f_n\|_X\leq C$, then $f\in X$ with $\|f\|_X\leq C$.
\item $\mathcal L$ leaves $X$ invariant and is bounded for $\|.\|_X$
\item There exists $M_z>0$ such that $$\sup_n\{\|\mathcal L^n(f)\|_\infty, f\in X, \|f\|_\infty\leq 1 \}\leq M_z$$
\item There exists an integer $n_0$ and two constants $0<a<1$ and $b\geq 0$ such that $\|\mathcal L^{n_0}(f)\|_X\leq a\|f\|_X+b\|f\|_\infty$ for all $f\in X$.
 \item If $Y$ is bounded in $X$, then $\mathcal L^{n_0}(Y)$ has compact closure in $\mathcal C^0(J)$.
\end{enumerate}
Then $\mathcal L$ is quasi compact on $X$: The spectrum is the union of finitely many isolated complex values which are eigenvalues with strictly dominating modulus and the essential spectrum is contained in an open disk of radius strictly smaller than the modulus of the eigenvalues.

\end{theorem}

We consider the operator $\mathcal L_{z,\beta, V}$. We want to apply previous theorem. In order to do so, we need to check the hypothesis.

Consider the subspace $X$ of the Hölder continuous functions $g$ from $J$ to $\mathbb R$ of exponent $\alpha$ defined by the following. 
$$|g(x)-g(y)|\leq Cd(x,y)^\alpha, \forall x,y\in J$$
Consider the following norm
$$\|g\|_X=\sup_J |g(x)|+\sup_{x\neq y\in J} \frac{|g(x)-g(y)|}{d(x,y)^{\alpha}}$$

Remark that $\|g\|_X$ defines a norm, and makes of $X$ a Banach space.
We will prove, using Theorem \ref{ionescu}, that $\lambda_z$, the spectral value, is an eigenvalue.

\begin{lemma}Let $z>z_c$, then the hypothesis of preceding theorem are satisfied if $X$ is the set of Holder functions for operator $\frac{1}{\lambda_z}\mathcal L_z$.
\end{lemma}

\begin{proof}
We check the different hypotheses. 
\begin{enumerate}
\item Assume $\|f_n\|_X\leq C$ for a sequence of $\alpha$ Hölder functions, then we have $|f_n(x)-f_n(y)|\leq C\|x-y\|^\alpha$. We deduce that $f$ is $\alpha$ Hölder, thus in $X$.
\item We prove that $X$ is invariant by $\CL_Z$:
$$\CL_z(f)(x)=\sum_y e^{\beta (S_k Vy)}f(y)e^{-kz}=\sum_y e^{\beta (S_kVy)}f(yx)e^{-kz}$$
Consider $x, x'\in J$, then the set $\{y, \exists n\in\mathbb N, \sigma^ny=x\}$ is in bijection with the set $\{y', \sigma^ny'=x'\}$ since $\Sigma_J$ is a SFT: Indeed such $y$ can be written $w_J..x=px$ and $p$ contains only one occurence of $w_J$. Thus we can write with $k=\tau(y)$:
\begin{align*}
|\CL_zf(x)-\CL_zf(x')|&\leq \sum_y e^{\beta (SkVy)-kz}|f(yx)-f(yx')|\\
 &\leq \sum_y e^{\beta (S_kVy)-kz}||f||_X d(yx,yx')^{\alpha}\\
 &\leq \|f\|_X\big(\sum_y \frac{e^{\beta(S_kVy)-kz}}{2^{|y|\alpha}}\big)d(x,x')^{\alpha}
\end{align*}
Remak that the sum is finite since $z>z_C$.
\item We know that $\CL_z(1\!\!1_J)$ is a constant function equal to $\lambda_z$. Then we have
$$\|\CL_z^nf\|_\infty\leq \|f\|_\infty.\|\CL_z^n(1\!\!1_J)\|$$
\item We remark that $|y|\geq 1$ if $y\in \sigma^{-1}(x)$. 
By the previous inequality
$$C_{\CL_zf}\leq C_f \sum_y \frac{e^{\beta (S_kVy)}}{2^{|y|\alpha}}\leq \frac{\lambda_z}{2^\alpha}C_f$$
Thus the condition is fulfilled for the operator $\frac{1}{\lambda_z}\CL_z$

We finish with the inequality $\|f\|_X=C_f+\|f\|_\infty$. 
\item We use Ascoli theorem.
\end{enumerate}
\end{proof}

\subsection{Technical lemmas}
We deduce from this lemma:
\begin{corollary}\label{cor-spectre}
For all $z>z_c$, the operator admits a spectral radius $\lambda_z$ which is an eigenvalue and equal to $\mathcal L_z(\BBone_J)$. If $(J,f)$ is mixing, then the eigenspace associated to $\lambda_z$ is of dimension one.
\end{corollary}
\begin{proof}
There is a finite number of non essential eigenvalues thus the supremum exists.

The function $\BBone_J$ is positive, and is an eigenvector associated to some eigenvalue, denoted $\lambda$. Let $\mu$ be another eigenvalue associated to $f$. Then consider the function $||f||_\infty\BBone_J-f$. It is a positive function, thus by definition of $\mathcal L_z$, its image is positive. We deduce $||f||_\infty \lambda \BBone_J-|\mu| f\geq 0$, thus $\lambda\geq |\mu|$.
We conclude that $\lambda_z$ which is the greatest eigenvalue is equal to $\mathcal{L}_z(\BBone_J)$.
If $(J,f)$ is mixing, then by \cite{Baladi} we have that the eigenspace associated to $\lambda_z$ is of dimension one.

\end{proof}

\begin{lemma}\label{lem-pression-retour}

For all $z>z_c$, there exists a unique equilibrium measure $m_z$ for $(J,f,S_{\tau}V-\tau z)$ of pressure $\log \lambda_z$.
The same result is true for $z=z_c$ if $\CL_{z_C}(\BBone_J)$ is finite.
\end{lemma}
\begin{proof}
We  consider  $z>z_c$, then Corollary \ref{cor-spectre} shows that there exists a measure $m_z$ such that $$\CL_z(\BBone_J)=\lambda_z \BBone_J, \CL_z^*(m_z)=\lambda_zm_z.$$ 
Remark that $m_z$ is a measure by positivity of the operator.
Then we consider the measure defined by $\BBone_J m_z$. It is clear that it is an invariant measure on $J$, and by \cite{Baladi} it is the unique measure of maximal pressure.

Consider $x$ and a cylinder $C_p(x)$  of length $p$ for $(J,f)$ which contains $x$. Then we compute
$m_z(C_p(x))=\int 1\!\!1_{C_p(x)}dm_z$. By definition of $m_z$ we deduce 

$$
m_z(C_p(x))=\frac{1}{\lambda_z}\int \CL(1\!\!1_{C_p(x)})dm_z=\frac{1}{\lambda_z^p}\int \CL^p(1\!\!1_{C_p(x)})dm_z$$

$$\CL_z(1\!\!1_{C_p})(x)=\sum_n\sum_y e^{\beta S_n V(y)} e^{-nz}=\sum_n\sum_ye^{\beta S_{\tau(y)}V(y)}e^{-\tau(y)z}$$
It is a constant function since we need to find all $u$ of length $n$ which start with $w$ in order to have $y=ux$ with $u$ of length $n$.
We iterate and obtain

$$\CL^p_z(1\!\!1_{C_p(x)})(x)=\sum_{k_i}\sum_{y, \sigma^{k_1+\dots+k_p}(y)=x}e^{\beta S_{k_1+\dots+k_p}V(y)-(k_1+\dots+k_p)z} $$
Thus 
\begin{equation}\label{eq-pression-m}
\lambda_z^p m_z(C_p(x))=\CL^p(1\!\!1_C)(x)=\sum_{k_i}\sum_{y, \sigma^{k_1+\dots+k_p}(y)=x}e^{\beta S_{k_1+\dots+k_p}V(y)-(k_1+\dots+k_p)z}.
\end{equation}

Now we use that $m_z$ is a Gibbs measure, and thus up to some multiplicative constant we obtain

$$\lambda_z^p m_z(C_p(x))\thickapprox \sum_{k_i}\sum_{y, \sigma^{k_1+\dots+k_p}(y)=x}m_z([y])e^{pP(J,f)} $$

$$\lambda_z^p m_z(C_p(x))\thickapprox  e^{p P(J,f)} m_z(C_p(x))$$

$$pP(J,f)=p\log \lambda_z+cst$$
since it is true for all $p$, we deduce
$$\log \lambda_z=P(J,f)=h_{m_z}+\int_J (\beta S_{\tau}V -\tau z)d m_z$$

Thus $\log\lambda_z$ is the pressure of the measure $m_z$.
\end{proof}

\subsection{Last part of the proof}
Remark that $z\mapsto \lambda_z$ is decreasing.

\begin{lemma}\label{lem-mesure-mu}
There exists a measure $\mu_z$ invariant for the system such that $(\mu_z)_{|J}=m_z$ if and only  if there exists $x\in J$ such that $\mathcal L_z(\tau)(x)$ converges. It is the case for $z>z_c$.
\end{lemma}
\begin{proof}
From the invariant measure $m_z$ of $(J,f)$ we want to construct a measure $\mu_z$ on the full shift. By a classical result, a necessary and sufficient condition is $\int \tau_J dm_z<\infty$,

The problem is reduced to the convergence of $\int \tau dm_z$. By definition of $m_z$ we have for all $f\in \mathcal C(J,\mathbb R)$, $\int_J \CL_z(f)dm_z=\lambda_z\int_J f dm_z$. We apply the equality for $f=\tau$ (or to a sequence of continuous functions which converges to $\tau$)and use the fact that $\CL_z(\tau)$ is a constant function.
It is the same as the convergence  $\CL_z(\tau)(x)$ for all $x\in J$. 

By definition, $\mathcal L_z(\tau)=\sum_n\sum_{\tau(y)=n}e^{S_nV(y)}ne^{-nz}$, 
thus it is the derivative with respect to $z$ of $-\mathcal L_z(\BBone_J)$. We deduce the convergence if $z>z_C$ by the hypothesis on $\CL_{z}(\BBone_J)$.
\end{proof}

\begin{lemma} If $z>z_c$ and $\mu(J)>0$, we obtain
$P(\Sigma,\sigma, \mu_z, V)=z+\mu_z(J)\log\lambda_z$ for $\beta\geq \beta_0$.
\end{lemma}
\begin{proof}
By Lemma \ref{lem-pression-retour}, $\log \lambda_z$ is the pressure of the system $(J, f)$ with potential $S_rV-\tau z/\beta$. Moreover $m_z$ is the equilibrium state. Thus we have
$$\log\lambda_z=h_{m_z}+\int_J (\beta S_\tau V-\tau z)dm_z.$$

Abramov's formula give us $h_{m}=\frac{h_\mu}{\mu(J)}$. Moreover we have $\int_J \beta S_\tau Vdm=\frac{1}{\mu_z(J)} \int_X \beta Vd\mu$. With Lemma \ref{lem-mesure-mu} we deduce 
$$h_{\mu_z}+\int _X\beta V d\mu_z=z+\mu_z(J)\log \lambda_z$$
\end{proof}

\begin{corollary}
If $z\geq P$, then $\log\lambda_z\leq 0$.
\end{corollary}
\begin{proof}
The left term has for upper bound $P$ since $\mu_z$ is an invariant measure for the global system. We deduce  $\lambda_z\leq 1$ for $z\geq P$.
\end{proof}

\begin{lemma}
We have $P(\beta) \geq 0$ for all $\beta\geq 0$.
\end{lemma}
\begin{proof}
Consider the measure $\mu_{\mathbb K}$, and the fact that $P$ is the supremum over all the invariant measures.
\end{proof}

\subsection{Proof of Theorem \ref{thm:meta-phase}.} 

We assume that $\mu(J)\neq 	0$ for an equilibrium measure $\mu$ of $(\Sigma,\sigma,\beta V)$. 

By Lemma \ref{lem-zero} we find a measure $m$ for $(J,f)$ of zero pressure. Since $P\geq 0$, by Lemma \ref{lem-pression-retour} there exists $m_P$ an equilibrium measure for $S_\tau V-\tau P$ of pressure $\log\lambda_P$. 
Thus the pressure of $m_P$ is bigger thant the pressure of $m$. Moreover since $z\mapsto \log \lambda_z$ is decreasing, thus $\log \lambda_P\leq \log\lambda_z<0$, which is a contradiction.


Therefore by the hypothesis of uniformity on $J$ of $\beta_0$, no equilibrium state gives positive weight to any cylinder which does not intersect $\mathbb{K}$, which means that any equilibrium state is supported into $\mathrm{TM}$. Now, we recall that $\mathbb{K}$ is uniquely ergodic, thus there is only one equilibrium state.

The theorem is proved.

\bibliographystyle{alpha}
\bibliography{biblioBCHL}
\end{document}